\documentclass[12pt]{article}

\begin{document}

\begin{titlepage}

 \vspace*{-2cm}

\vspace{.5cm}

\begin{centering}

\huge{Harmonic analysis of the functions $\tilde{\Delta}(x)$ and
$N(T)$  }

\vspace{.5cm}

\large  {Jining Gao }\\

\vspace{.5cm}

Department of Mathematics, Shanghai Jiaotong University, Shanghai
,P. R. China

\vspace{.5cm}
\begin{abstract}

In this paper,  under the Riemann hypothesis, we study the Fourier
analysis about the functions $\tilde{\Delta}(x)$ and $N(T)$ .

\end{abstract}

\end{centering}

\end{titlepage}

\pagebreak

\def\lh{\hbox to 15pt{\vbox{\vskip 6pt\hrule width 6.5pt height 1pt}
  \kern -4.0pt\vrule height 8pt width 1pt\hfil}}
\def\blob{\mbox{$\;\Box$}}
\def\qed{\hbox{${\vcenter{\vbox{\hrule height 0.4pt\hbox{\vrule width
0.4pt height 6pt \kern5pt\vrule width 0.4pt}\hrule height
0.4pt}}}$}}

\newtheorem{theorem}{Theorem}
\newtheorem{lemma}[theorem]{Lemma}
\newtheorem{definition}[theorem]{Definition}
\newtheorem{corollary}[theorem]{Corollary}
\newtheorem{proposition}[theorem]{Proposition}
\newcommand{\proof}{\bf Proof.\rm}

\section{INTRODUCTION}

 Riemann hypothesis has been studied in many different ways, in this
 paper, we will try to use somewhat new angles to study RH.
 Most results of this paper are obtained under the RH.
 As we know, Guinand formula is a representation of $N(T)$
 which is the distribution function of  Riemann zeros in
 term of series of the  prime number powers, although Guinand
 formula \cite{G}is a result under the assumption of RH,it provides an
 explicit method to figure out all non trivial Riemann zeros.
 Actually, this fact is far from trivial because once we have
 prime number representation of $N(T)$ (Guinand formula) at hand,
 we can immediately restore a function via the distribution of it's
 zeros,so Guinand formula is equivalent to RH and we will prove it
 in the first section. Since Guinand formula is very important in
 this paper and Guinand original proof is complicated and full of
 the favor of harmonic analysis, we will first of all give another
 simple and elementary proof based on the lemma, which is the
 ground stone of this paper, besides ,our new proof gives out
 a stronger conclusion  than the original statement of Guinand
 formula. This stronger result will help us to check the truth
 of RH  much more efficiently.
 In the second section we rewrite Guinand formula and
 Riemann-Mangoldt formula as two integral equations of two
 "functional variables" $\tilde{\Delta}(x)$ and $S(T)$, which
 seems to imply Guinand formula and  Riemann-Mangoldt formula are
 reciprocal to each other and such integral representation will be
 used in the 4th section.
In the third section ,first of all, we derive an elementary formula
based on functional equation of Riemann zeta function and lemma.
This formula provides infinitely many non trivial integral equations
of $N(T)$, also, we use the elementary formula to prove  a theorem
which claims $\mid\tilde{\Delta}(x)\mid$ has a non-zero measurement
of a positive lower bound.

\section{Guinand formula with an error term and it's inverse theorem}
In this section, we will give out a proof of Guinand formula with
the uniformly convergent error term, besides, we also give out an
inverse theorem of Guinand formula. First of all we need following
notations and formulas which will be used throughout this
paper\cite{E},\cite{KV}.

{\bf Chebyshev function}
\begin{eqnarray}
\psi(x)=\sum_{n< x}\Lambda(n)\nonumber
\end{eqnarray}
Where the Von Mangoldt function $\Lambda(n)=logp$ if $n=p^k$ for
some $k$ and some prime number $p$ ,$\Lambda(n)=0$ otherwise.

\begin{theorem}{\bf ( Mangoldt and Riemann explicit formula)}
\begin{eqnarray}
\psi(x)=x-lim_{T\rightarrow \infty}\sum_{\left|Im\rho \right|< T}
\frac{x^{\rho}}{\rho}-log(2\pi)+\sum_{n=1}^{\infty}\frac{x^{-2n}}{2n}\nonumber
\end{eqnarray}
Where $\rho$ runs through all non-trivial Riemann zeros.
\end{theorem}

\begin{theorem} \cite{KV}
\begin{eqnarray}
\psi(x)=x-\sum_{\left|Im\rho \right|< T}
\frac{x^{\rho}}{\rho}+O(\frac{xlog^{2}x}{T})\nonumber
\end{eqnarray}
\end{theorem}
we set
\begin{eqnarray}
\tilde{\psi}(x)=\frac{x^2}{2}-\sum_{\rho}\frac{x^{\rho+1}}{\rho(\rho+1)}-xln(2\pi)
-\sum_{n=1}^{\infty}\frac{x^{-2n+1}}{2n(2n-1)}\label{a}
\end{eqnarray}
It's easy to prove that when $x$ is not equal to any
integer,$\tilde{\psi}(x)$ is differentialble and it's derivative is
just $\psi(x)$ and it's continuous when $x>0$ \cite{E}. We set
$\Delta(x)=\psi(x)-x$ and
$\tilde{\Delta}(x)=\tilde{\psi}(x)-\frac{x^2}{2}$ .

 The following lemma will is important for  deducing the Guinand
 formula with the error term.

 \begin{lemma}
when $s\neq \rho$,
\begin{eqnarray}
\frac{\zeta'}{\zeta}(s)=-\sum_{n<X}\frac{\Lambda(n)}{n^s}+\psi(X)X^{-s}+s\tilde{\psi}(X)X^{-s-1}- \nonumber\\
\frac{s(s+1)}{2(s-1)}X^{1-s}+\sum_{\rho}\frac{s(s+1)X^{\rho-s}}{\rho(\rho+1)(s-\rho)}+\sum_{n\geq
1}\frac{s(s+1)X^{-2n-2s}}{2n(2n-1)(s+2n)}\label{a0}
\end{eqnarray}
\end{lemma}

\begin{proof}
Let $$f_{X}(s)=\sum_{n<X}\frac{\Lambda(n)}{n^s}$$ Using integration
by parts twicely, we have that
\begin{eqnarray}
f_{X}(s)=\int_{1}^{X}x^{-s}d\psi(x)\nonumber
=\psi(X)X^{-s}-\int_{1}^{X}\psi(x)dx^{-s}\nonumber
=\psi(X)X^{-s}+s\int_{1}^{X}\psi(x)x^{-s-1}dx\nonumber\\
=\psi(X)X^{-s}+s\int_{1}^{X}x^{-s-1}d\tilde{\psi}(x)\nonumber\\
=\psi(X)X^{-s}+
s\tilde{\psi}(X)X^{-s-1}+s(s+1)\int_{1}^{X}\tilde{\psi}(x)x^{-s-2}dx\label{b}
\end{eqnarray}

and by  formula \ref{a},we can further get
\begin{eqnarray}
\int_{1}^{X}\tilde{\psi}(x)x^{-s-2}dx=\int_{1}^{X}\frac{\frac{x^2}{2}-\sum_{\rho}\frac{x^{\rho+1}}{\rho(\rho+1)}-ln(2\pi) x-\sum_{n=1}^{\infty}\frac{x^{-2n+1}}{2n(2n-1)}}{x^{s+2}}dx\nonumber\\
=\frac{1}{2}(\frac{X^{1-s}}{1-s}-\frac{1}{1-s})-\sum_{\rho}\frac{1}{\rho(\rho+1)}(\frac{X^{\rho-s}}{\rho-s}-\frac{1}{\rho-s})\nonumber\\
+ln(2\pi)(\frac{X^{-s}}{s}-\frac{1}{s})+\sum_{n\geq
1}\frac{1}{2n(2n-1)}(\frac{X^{-s-2n}}{s+2n}-\frac{1}{s+2n})\nonumber
\end{eqnarray}
We collect all above terms as two groups $J_{X}(s)$ and
$I(s)$,obviously,
$$I(s)=\frac{1}{2}\frac{1}{s-1}-\sum_{\rho}\frac{1}{\rho(\rho+1)}\frac{1}{s-\rho}-\frac{ln(2\pi)}{s}-\sum_{n\geq 1}\frac{1}{2n(2n-1)}\frac{1}{s+2n}$$
By formula \ref{b} and using the notations  $J_{X}(s)$ and $I(s)$,
we get
\begin{eqnarray}
f_{X}(s)=\psi(X)X^{-s}+
s\tilde{\psi}(X)X^{-s-1}+s(s+1)J_{X}(s)+s(s+1)I(s)\label{c}
\end{eqnarray}
and
$$s(s+1)I(s)=\frac{s(s+1)}{2(s-1)}-\sum_{\rho}\frac{s(s+1)}{\rho(\rho+1)(s-\rho)}-(s+1)ln(2\pi)-\sum_{n\geq
1}\frac{s(s+1)}{2n(2n-1)(s+2n)}$$ By the following identity,
\begin{eqnarray}
\frac{s(s+1)}{z(z+1)(s-z)}=\frac{s}{z(z+1)}+\frac{1}{s-z}+\frac{1}{z}\nonumber
\end{eqnarray}
we have that
\begin{eqnarray}
s(s+1)I(s)=-\frac{\zeta'}{\zeta}(s)+as+b \label{d}
\end{eqnarray}
where $a,b$ are some constants which can be determined immediately.
According \ref{c} and \ref{d} ,we get a new representation of
$\frac{\zeta'}{\zeta}(s)$ when $s\neq \rho$ as follows:
\begin{eqnarray}
\frac{\zeta'}{\zeta}(s)=-\sum_{n<X}\frac{\Lambda(n)}{n^s}+\psi(X)X^{-s}+s\tilde{\psi}(X)X^{-s-1}\nonumber\\
-\frac{s(s+1)}{2(s-1)}X^{1-s}+\sum_{\rho}\frac{s(s+1)X^{\rho-s}}{\rho(\rho+1)(s-\rho)}+\sum_{n\geq
1}\frac{s(s+1)X^{-2n-s}}{2n(2n-1)(s+2n)}+as+b\label{e}
\end{eqnarray}
Since when
$Res>1$,$$\frac{\zeta'}{\zeta}(s)=-\sum_{n=1}^{\infty}\frac{\Lambda(n)}{n^s}$$
Let $X\rightarrow \infty$ on the right side of \ref{e} when $Res>1$,
we immediately get $a=b=0$, that follows our theorem.
\end{proof}
\newline
As we know, $$log\zeta(s_0)=\int_{2}^{s_0}\frac{\zeta'}{\zeta}(s)ds
\label{f}$$,where the integral path is a positive orient half
rectangle with vertices $2,2+iT, \sigma+iT$ and $s_{0}=\sigma+iT$
$s_0 \neq \rho$. Taking this complex integral  on both sides of
\ref{e}, we directly get following theorem:

\bigskip
\begin{theorem}
When $s_0\neq \rho $, we have that
\begin{eqnarray}
log\zeta(s_0)=\sum_{n<X}\frac{\Lambda(n)}{(logn) n^{s_0}}-\frac{\Delta(X)}{(logX)X^{s_0}}-\frac{\tilde{\Delta}(X)}{(logX)X^{s_{0}+1}}(s_{0}+\frac{1}{lnX})\nonumber\\
+\int_{2}^{s_0}\frac{X^{1-s}}{1-s}ds+\tilde{J}_{X}(s_0)+C_{0}\label{h}
\end{eqnarray}
Where
\begin{eqnarray}
\tilde{J}_{X}(s_0)=\int_{2}^{s_0}s(s+1)J(X)ds\nonumber
=-\frac{1}{lnX}\sum_{\rho}\frac{1}{\rho(\rho+1)}[\frac{s_{0}(s_{0}+1)X^{\rho-s_0}}{s_{0}-\rho}\nonumber\\
-\int_{2}^{s_0}X^{\rho-s}(\frac{2s+1}{s-\rho}-\frac{s^{2}+s}{(s-\rho)^2})ds]\nonumber\\
-\frac{1}{lnX}\sum_{n\geq 1}\frac{1}{2n(2n-1)}[\frac{
s_{0}(s_{0}+1)X^{-2n-s}}{s+2n}-\int_{2}^{s_0}X^{-2n-s}(\frac{2s+1}{s+2n}-\frac{s^{2}+s}{(s+2n)^2})ds]\nonumber
\end{eqnarray}
and $C_{0}$  is a real constant.
\end{theorem}
\begin{proof}
Using integration by parts and collecting all terms containing
$X^{1-s}$, we immediately get above results.
\end{proof}
setting $s_0=\frac{1}{2}+iT$ in the formula \ref{h} and taking
imaginary parts on both sides,   we have that
\begin{theorem}
If the Riemann hypothesis is true, and $\delta$ is the distance
between $T$ and the coordinate of the nearest Riemann zero,we have
\begin{eqnarray}
\pi
S(T)=-\sum_{n<X}\frac{\Lambda(n)sin(Tlogn)}{\sqrt{n}}+\frac{\Delta(X)sin(TlogX)}{\sqrt{X}(logX)}+Im(\int_{2}^{\frac{1}{2}+iT}\frac{X^{1-s}}{1-s}ds)+
O(\frac{T^3}{\delta^{2}lnX})\label{a5}
\end{eqnarray}
 in the limit language, we have
\begin{eqnarray}
\pi
S(T)=-lim_{X\rightarrow \infty}[\sum_{n<X}\frac{\Lambda(n)sin(Tlogn)}{\sqrt{n}(logn)}\nonumber\\
-
\frac{\Delta(X)sin(TlogX)}{\sqrt{X}(logX)}-Im(\int_{2}^{\frac{1}{2}+iT}\frac{X^{1-s}}{1-s}ds)]\label{i}
\end{eqnarray}
\end{theorem}
From now on ,we will prove formula \ref{i} is the same as Guinand
formula. To achieve it, we need to make some simplification as
follows:

Let's first simplify the term
$$Im(\int_{2}^{\frac{1}{2}+iT}\frac{X^{1-s}}{1-s}ds)$$ ,Let's transform the original integral path which is half rectangle with vertices $2,2+iT, \frac{1}{2}+iT$ to another
half rectangle with vertices $2,\frac{1}{2},\frac{1}{2}+iT$ and
orient is clockwise, we get
\begin{eqnarray}
\int_{2}^{\frac{1}{2}+iT}\frac{X^{1-s}}{1-s}ds=i\int_{0}^{T}\frac{X^{\frac{1}{2}-it}}{\frac{1}{2}-it}dt+\int_{\Gamma_r}\frac{X^{1-s}}{1-s}ds\label{j}
\end{eqnarray}
Where $\gamma_{r}=[\frac{1}{2},1-r]\cup S_{r}\cup[1+r,2]$ and $S_r$
is upper half semi-circle with radius $r$ and centered at  $z=1$
Taking imaginary part on both side of \ref{j},we get the first term
of right hand side is equal to
$$\sqrt{X}\int_{0}^{T}\frac{2cos(tlogX)+4tsin(tlogX)}{1+4t^2}dt$$ which is set to be $f_{1}(T,X)$
For the second term of right side of \ref{j}, we have that
\begin{eqnarray}
Im(\int_{\Gamma_r}\frac{X^{1-s}}{1-s}ds)=Im(\int_{J_r}\frac{X^{1-s}}{1-s}ds)\nonumber\\
=lim_{r\rightarrow
0}Im(\int_{J_r}\frac{X^{1-s}}{1-s}ds)=-\pi\label{b2}
\end{eqnarray}

Let's pick up the second term of righ side of \ref{a1} i.e.
$\int_{1}^{X}\frac{sin(Tlogt)}{\sqrt{t}logt}dt$ and set it to be
$f_{2}(T,X)$, Since $\frac{sin(Tlogt)}{\sqrt{t}logt}$ is continously
differentiable with respect to $T$, we have that
\begin{eqnarray}
\frac{\partial f_2}{\partial T}=\int_{1}^{X}\frac{cos(Tlogt)}{\sqrt{t}}dt\nonumber\\
=\int_{0}^{lnX}e^{\frac{1}{2}u}cos(Tu)du
=\frac{\sqrt{X}}{1+4T^2}[2cos(TlogX)+4Tsin(TlogX]-\frac{2}{1+4T^2}\label{a2}
\end{eqnarray}

in which we have used the substituation $u=logt$

We can also notice that $$\frac{\partial f_1}{\partial
T}=\frac{\sqrt{X}}{1+4T^2}[2cos(TlogX)+4Tsin(TlogX]$$ By \ref{a2},
we get
\begin{eqnarray}
\frac{\partial f_{1}}{\partial T}-\frac{\partial f_{2}}{\partial
T}=\frac{2}{1+4T^2}\nonumber
\end{eqnarray}

Thus,
\begin{eqnarray}
f_{1}(T,X)-f_{2}(T,X)=\int_{0}^{T}\frac{2}{1+4t^2}dt\label{a3}\\
=arctan2T\nonumber
\end{eqnarray}
Consequently,by \ref{j}, \ref{b2},\ref{a3}
\begin{eqnarray}
Im(\int^{\frac{1}{2}+iT}\frac{X^{1-s}}{1-s}ds)-\int_{1}^{X}\frac{sin(Tlogt)}{\sqrt{t}logt}dt
=arctan2T-\pi\label{c2}
\end{eqnarray}
Let's single out the term $$\frac{sin(TlogX)}{logX}\left\{\sum_{n<
X}\Lambda(n)n^{-\frac{1}{2}}-2X^{\frac{1}{2}}\right\}$$ in right
hand of \ref{a1} and get it simplified as follows:

\begin{eqnarray}
\frac{sin(TlogX)}{logX}\left\{\sum_{n< X}\Lambda(n)n^{-\frac{1}{2}}-2X^{\frac{1}{2}}\right\} \nonumber\\
=\frac{sin(TlogX)}{logX}[\int_{1}^{X}x^{-\frac{1}{2}}d\psi(x)-2\sqrt{X} ]\nonumber\\
=\frac{sin(TlogX)}{logX}[\psi(X)X^{-\frac{1}{2}}+\frac{1}{2}\int_{1}^{X}\psi(x)x^{-\frac{3}{2}}dx-2\sqrt{X}]\nonumber\\
=\frac{\Delta(X)sin(TlogX)}{\sqrt{X}(logX)}+\frac{sin(TlogX)}{2logX}\int_{1}^{X}\Delta(x)x^{-\frac{3}{2}}dx\label{b4}
\end{eqnarray}

and by the theorem 2 , we have that
\begin{eqnarray}
\int_{1}^{X}\Delta(x)x^{-\frac{3}{2}} dx= -\sum_{\rho}\frac{X^{\rho-\frac{1}{2}}-1}{\rho(\rho-\frac{1}{2})}+2ln(2\pi)(X^{-\frac{1}{2}}-1)\nonumber\\
-\sum_{n\geq
1}\frac{X^{-2n-\frac{1}{2}}-1}{2n(2n+\frac{1}{2})}=O(1)\nonumber
\end{eqnarray}
With formula \ref{b4}, we have
\begin{eqnarray}
\frac{sin(TlogX)}{logX}[\sum_{n\leq
X}\Lambda(n)n^{-\frac{1}{2}}-2X^{\frac{1}{2}}]=\frac{\Delta(X)sin(TlogX)}{\sqrt{X}(logX)}+O(\frac{1}{logX})\label{c3}
\end{eqnarray}

Since
\begin{eqnarray}
N(T)=\frac{1}{\pi}arg\xi(\frac{1}{2}+iT)\nonumber\\
=\frac{1}{\pi}args(s-1)\pi^{-\frac{s}{2}}\Gamma(\frac{s}{2})\zeta(s)|_{s=\frac{1}{2}+iT}\nonumber\\
=\frac{1}{\pi}arg(-\frac{1}{4}-T^{2})-\frac{Tln \pi}{2\pi}+\frac{1}{\pi}arg\Gamma(\frac{1}{4}+\frac{iT}{2})+S(T)\nonumber\\
=1-\frac{Tln\pi}{2\pi}+\frac{1}{\pi}arg\Gamma(\frac{1}{4}+\frac{iT}{2})+S(T)\label{b3}
\end{eqnarray}

Whenever $T$ is not equal to any cordinates of some Riemann zeros,we
can rewrite Guinand formula \ref{a1} as follows
\begin{eqnarray}
\pi S(T)=F_{X}(T)+arctan(2T)-\pi+\frac{1}{2}arg\Gamma(\frac{1}{2}+iT)-arg\Gamma(\frac{1}{4}+\frac{iT}{2})\nonumber\\
-\frac{Tln2}{2}-\frac{1}{4}arctan(sinh\pi T)\label{c1}
\end{eqnarray}
Where
\begin{eqnarray}
F_{X}(T)=-lim_{X\rightarrow  \infty}[\sum_{n\leq X}\Lambda(n)\frac{sin(Tlogn)}{\sqrt{n}logn}-\int_{1}^{X}\frac{sin(Tlogt)}{\sqrt{t}logt}dt\nonumber\\
-\frac{sin(TlogX)}{logX}[\sum_{n\leq
X}\Lambda(n)n^{-\frac{1}{2}}-2X^{\frac{1}{2}}]]\nonumber
\end{eqnarray}
Using equation \ref{c1}(Guinand formula)  minus  equation \ref{i}
and notice \ref{c2} and \ref{c3}, we get that
\begin{eqnarray}
0=\frac{1}{2}arg\Gamma(\frac{1}{2}+iT)-arg\Gamma(\frac{1}{4}+\frac{iT}{2})
-\frac{Tln2}{2}-\frac{1}{4}arctan(sinh\pi T)\label{d1}
\end{eqnarray}
When $T$ is not cordinates of some Riemann zeros. We set right side
of \ref{d1} to be $d(T)$,then we just need to prove that $d(T)\equiv
0$  when  $T>0$. Let's show it as follows: By rewriting
$$\frac{Tln2}{2}=arg2^{\frac{iT}{2}}$$ and $$arctan(sinh\pi
T)=arg(1+isinh\pi T)$$, we have that
\begin{eqnarray}
4d(T)=arg\frac{\Gamma^{2}(\frac{1}{2}+iT)}{\Gamma^{4}(\frac{1}{4}+\frac{iT}{2})4^{iT}(1+isinh\pi T)}\nonumber\\
=Imlog\frac{\Gamma^{2}(\frac{1}{2}+iT)}{\Gamma^{4}(\frac{1}{4}+\frac{iT}{2})4^{iT}(1+isinh\pi T)}\nonumber\\
=\frac{1}{2i}[log\frac{\Gamma^{2}(\frac{1}{2}+iT)}{\Gamma^{4}(\frac{1}{4}+\frac{iT}{2})4^{iT}(1+isinh\pi
T)}-log\frac{\Gamma^{2}(\frac{1}{2}-iT)}{\Gamma^{4}(\frac{1}{4}-\frac{iT}{2})4^{-iT}(1-isinh\pi
T)}]\nonumber
\end{eqnarray}
Let $s=iT$, then $$sinh\pi T=-isin\pi s$$ and
\begin{eqnarray}
4d(T)=\frac{1}{2i}log\frac{\Gamma^{2}(\frac{1}{2}+s)\Gamma^{4}(\frac{1}{4}-\frac{s}{2})(1-sin\pi
s)}{\Gamma^{2}(\frac{1}{2}-s)\Gamma^{4}(\frac{1}{4}+\frac{s}{2})4^{2s}(1+sin\pi
s)}\label{g}
\end{eqnarray}
Let
$$g(s)=\frac{\Gamma^{2}(\frac{1}{2}+s)\Gamma^{4}(\frac{1}{4}-\frac{s}{2})(1-sin\pi
s)}{\Gamma^{2}(\frac{1}{2}-s)\Gamma^{4}(\frac{1}{4}+\frac{s}{2})4^{2s}(1+sin\pi
s)}$$ then $g(s)$ is a meromorphic function in the whole complex
number plane and by the formula \ref{g},  $g(s)|_{s=iT}=1$
 when $T>0$ . For the convenience of factorizing $g(s)$, let's set $s=\frac{1}{2}-z$ and reset $f(z)=g(s)$
 we have that
\begin{eqnarray}
f(z)=\frac{\Gamma^{2}(1-z)\Gamma^{4}(\frac{1}{2}
z)sin^{2}(\frac{\pi}{2}
z)}{\Gamma^{2}(z)\Gamma^{4}(\frac{1}{2}-\frac{1}{2}
z)4^{1-2z}cos^{2}(\frac{\pi}{2} z) }
\end{eqnarray}
We just need to prove that $f(z)\equiv 1$ for any $z \in C$ , that
can be derived by the formula
$$\Gamma(z)\Gamma(1-z)=\frac{\pi}{sin(\pi z)}$$

With the formulas \ref{c3},\ref{c1},\ref{d1}, we can rewrite the
formula \ref{a5} as:
$$\pi S(T)=-\sum_{n<X}\frac{\Lambda(n)sin(Tlogn)}{\sqrt{n}}+\frac{\Delta(X)sin(TlogX)}{\sqrt{X}(logX)}+\int_{1}^{X}\frac{sin(Tlogy)}{\sqrt{y}logy}dy$$
$$+ arctan(2T)-\pi+O(\frac{T^3}{\delta^{2}lnX})\label{a6}$$

Furthermore, we can rewrite above formula as an integral equation,
first of all,

$$\sum_{n<X}\frac{\Lambda(n)sin(Tlogn)}{\sqrt{n}}=\int_{a}^{X}\frac{sin(Tlogy)}{\sqrt{y}logy}d\psi(y)\nonumber$$
$$=\int_{a}^{X}\frac{sin(Tlogy)}{\sqrt{y}logy}dy+\int_{a}^{X}\frac{sin(Tlogy)}{\sqrt{y}logy}d\Delta(y)\nonumber$$
$$=\int_{a}^{X}\frac{sin(Tlogy)}{\sqrt{y}logy}dy+\frac{\Delta(X)sin(TlogX)}{\sqrt{X}(logX)}-\frac{\Delta(a)sin(Tloga)}{\sqrt{a}(loga)}$$
$$-\int_{a}^{X}\frac{Tcos(Tlny)-sin(Tlny)(\frac{lny}{2}+1)}{y\sqrt{y}ln^{2}y}\Delta(y)dy$$

Where $1<a<2$.
\newline
 We substitute above formula into \ref{a6},we get that
$$S(T)=-\frac{1}{\pi}\int_{a}^{X}\frac{Tcos(Tlny)-sin(Tlny)(\frac{lny}{2}+1)}{y\sqrt{y}ln^{2}y}\Delta(y)dy$$
$$-\frac{1}{\pi}[\int_{a}^{1}\frac{sin(Tlogy)}{\sqrt{y}logy}dy-\frac{\Delta(a)sin(Tloga)}{\sqrt{a}(loga)}+arctan(2T)-\pi]\label{b5}$$

and

$$\frac{1}{\pi}\int_{a}^{X}\frac{Tcos(Tlny)-sin(Tlny)(\frac{lny}{2}+1)}{y\sqrt{y}ln^{2}y}\Delta(y)dy$$
$$=\frac{1}{\pi}\int_{a}^{X}\frac{Tcos(Tlny)-sin(Tlny)(\frac{lny}{2}+1)}{y\sqrt{y}ln^{2}y}d\tilde{\Delta}(y)$$
$$=\tilde{\Delta}(X)\frac{Tcos(TlnX)-sin(TlnX)(\frac{lnX}{2}+1)}{X\sqrt{X}ln^{2}X}-\tilde{\Delta}(a)\frac{Tcos(Tlna)-sin(Tlna)(\frac{lna}{2}+1)}{a\sqrt{a}ln^{2}a}$$
$$-\int_{a}^{X}\tilde{\Delta}(y)d\frac{Tcos(Tlny)-sin(Tlny)(\frac{lny}{2}+1)}{y\sqrt{y}ln^{2}y} $$

and
\begin{eqnarray}
\int_{a}^{X}\tilde{\Delta}(y)d\frac{Tcos(Tlny)-sin(Tlny)(\frac{lny}{2}+1)}{y\sqrt{y}ln^{2}y}=\int_{a}^{X}F(T,y)\tilde{\Delta}(y)dy\label{c4}
\end{eqnarray}

Where
$$F(T,y)=-\frac{T^{2}sin(Tlny)}{y^{\frac{5}{2}}lny}-\frac{2Tcos(Tlny)}{y^{\frac{5}{2}}lny}+\frac{3sin(Tlny)}{4y^{\frac{5}{2}}lny}$$
$$-\frac{2Tcos(Tlny)}{y^{\frac{5}{2}}ln^{2}y}+\frac{2sin(Tlny)}{y^{\frac{5}{2}}ln^{2}y}+\frac{2sin(Tlny)}{y^{\frac{5}{2}}ln^{3}y}\label{a7}$$

By \ref{a6},\ref{b5},\ref{c4} and when t is not the cordinate of a
Riemann zero, let $X\rightarrow \infty$, we have following integral
equation

\begin{eqnarray}
S(t)=-\frac{1}{\pi}\int_{a}^{\infty}F(t,y)\tilde{\Delta}(y)dy+g(a,t)\label{ie1}
\end{eqnarray}
Where
$$g(a,t)=-\frac{1}{\pi}[\int_{a}^{1}\frac{sin(tlogy)}{\sqrt{y}logy}dy-\frac{\Delta(a)sin(tloga)}{\sqrt{a}(loga)}$$
$$+\tilde{\Delta}(a)\frac{tcos(tlna)-sin(tlna)(\frac{lna}{2}+1)}{a\sqrt{a}ln^{2}a}
+arctan(2t)-\pi]$$, $1<a<2$
\bigskip
\section{Representing $\tilde{\Delta}(x)$ in term of $S(T)$ }
In this section, under the RH , we will represent
$\tilde{\Delta}(x)$ as an integral of $S(T)$ via Riemann-Von
Mangoldt formula.  Let $N(T)$ be a function counting the number of
non-trivial Riemann zeros whose imaginary is between $0$ and
$T$,under the RH,we can rewrite the formula \ref{a}in term of $N(T)$
as follows,
\begin{eqnarray}
 \tilde{\Delta}(x)=-\sum_{\rho}\frac{x^{\rho+1}}{\rho(\rho+1)}-xln(2\pi)
-\sum_{n=1}^{\infty}\frac{x^{-2n+1}}{2n(2n-1)}\nonumber \\\\
 =-\int_{0}^{\infty}[\frac{x^{\frac{3}{2}+it}}{(\frac{1}{2}+it)(\frac{3}{2}+it)}+\frac{x^{\frac{3}{2}-it}}{(\frac{1}{2}-it)(\frac{3}{2}-it)}]dN(t)+f(x)\nonumber\\\\
 =-2x^{\frac{3}{2}}\int_{0}^{\infty}\frac{(\frac{3}{4}-t^2)cos(tlnx)+2tsin(tlnx)}{4t^{2}+(\frac{3}{4}-t^2)^2}dN(t)+f(x)
\end{eqnarray}

Where$$f(x)=-xln(2\pi)
-\sum_{n=1}^{\infty}\frac{x^{-2n+1}}{2n(2n-1)}$$ Noticing
\ref{b3},set
$$g(t)=1-\frac{tln\pi}{2\pi}+\frac{1}{\pi}arg\Gamma(\frac{1}{4}+\frac{it}{2})$$
Thus we have
\begin{eqnarray}
\tilde{\Delta}(x)=-2x^{\frac{3}{2}}\int_{0}^{\infty}\frac{(\frac{3}{4}-t^2)cos(tlnx)+2tsin(tlnx)}{4t^{2}+(\frac{3}{4}-t^2)^2}dS(t)\nonumber\\\\
-2x^{\frac{3}{2}}\int_{0}^{\infty}\frac{(\frac{3}{4}-t^2)cos(tlnx)+2tsin(tlnx)}{4t^{2}+(\frac{3}{4}-t^2)^2}dg(t)+f(x)
\end{eqnarray}
Putting the last two terms together and setting it to be
$\tilde{f}(x)$, we get that
\begin{eqnarray}
\tilde{\Delta}(x)=-2x^{\frac{3}{2}}\int_{0}^{\infty}\frac{(\frac{3}{4}-t^2)cos(tlnx)+2tsin(tlnx)}{4t^{2}+(\frac{3}{4}-t^2)^2}dS(t)+\tilde{f}(x)
\end{eqnarray}
Using integration by parts and noticing $S(t)=O(logt)$, we have
\bigskip
\begin{eqnarray}
\int_{0}^{\infty}\frac{(\frac{3}{4}-t^2)cos(tlnx)+2tsin(tlnx)}{4t^{2}+(\frac{3}{4}-t^2)^2}dS(t)\nonumber\\
=-\frac{4S(0)}{3}-\int_{0}^{\infty}S(t)d\frac{(\frac{3}{4}-t^2)cos(tlnx)+2tsin(tlnx)}{4t^{2}+(\frac{3}{4}-t^2)^2}\label{e1}
\end{eqnarray}
From now on, we are going to evaluate the second term of \ref{e1} in
detail for the convenience of checking.

 $$\int_{0}^{\infty}S(t)d\frac{(\frac{3}{4}-t^2)cos(tlnx)+2tsin(tlnx)}{4t^{2}+(\frac{3}{4}-t^2)^2}$$
$$=\int_{0}^{\infty}S(t)\frac{[(\frac{3}{4}-t^2)cos(tlnx)+2tsin(tlnx)]'[4t^{2}+(\frac{3}{4}-t^2)^2]}{[4t^{2}+(\frac{3}{4}-t^2)^2]^2}dt$$
$$-\int_{0}^{\infty}S(t)\frac{[(\frac{3}{4}-t^2)cos(tlnx)+2tsin(tlnx)][4t^{2}+(\frac{3}{4}-t^2)^2]'}
{[4t^{2}+(\frac{3}{4}-t^2)^2]^2}dt$$
$$=\int_{0}^{\infty}S(t)\frac{[-2tcos(tlnx)-(\frac{3}{4}lnx) sin(tlnx)+t^{2}lnx sin(tlnx)+2sin(tlnx)+2tlnx
cos(tlnx)][t^4+\frac{5}{2}+\frac{9}{16}]}{[4t^{2}+(\frac{3}{4}-t^2)^2]^2}dt
$$
$$-\int_{0}^{\infty}S(t)\frac{[\frac{3}{4}cos(tlnx)-t^{2}cos(tlnx)+2tsin(tlnx)][4t^3+5t]}{[4t^{2}+(\frac{3}{4}-t^2)^2]^2}dt$$

$$=\int_{0}^{\infty}S(t)\frac{-2t^{5}cos(tlnx)-\frac{3}{4}t^{4}sin(tlnx)+t^{6}lnxsin(tlnx)+2t^{4}sin(tlnx)+2t^{5}lnxcos(tlnx)}{(t^4+\frac{5}{2}t^{2}+\frac{9}{16})^2}dt\\$$

$$+\int_{0}^{\infty}S(t)\frac{-5t^{3}cos(tlnx)-\frac{15}{8}t^{2}sin(tlnx)+\frac{5}{2}t^{4}lnxsin(tlnx)+5t^{2}sin(tlnx)+5t^{3}lnxcos(tlnx)}{(t^4+\frac{5}{2}t^{2}+\frac{9}{16})^2}dt$$

$$+\int_{0}^{\infty}S(t)\frac{-\frac{9}{8}tcos(tlnx)-(\frac{27}{64}lnx) sin(lnx)+\frac{9}{16}t^{2}lnx sin(tlnx)+\frac{9}{8}sin(tlnx)+\frac{9}{8}tlnx
cos(tlnx)}{(t^4+\frac{5}{2}t^{2}+\frac{9}{16})^2}dt$$

$$-\int_{0}^{\infty}S(t)\frac{3t^{3}-4t^{5}cos(tlnx)+8t^{4}sin(tlnx)+\frac{15}{4}tcos(tlnx)-5t^{3}cos(tlnx)+10t^{2}sin(tlnx)}{(t^4+\frac{5}{2}t^{2}+\frac{9}{16})^2}dt$$

To summarize, we get
\begin{eqnarray}
\tilde{\Delta}(x)=-2x^{\frac{3}{2}}\int_{0}^{\infty}K(x,t)S(t)dt+f(x)\label{ie2}
\end{eqnarray}
Where we still let
$$f(x)=-2x^{\frac{3}{2}}\int_{0}^{\infty}\frac{(\frac{3}{4}-t^2)cos(tlnx)+2tsin(tlnx)}{4t^{2}+(\frac{3}{4}-t^2)^2}dg(t)-xln(2\pi)
-\sum_{n=1}^{\infty}\frac{x^{-2n+1}}{2n(2n-1)}+\frac{8S(0)}{3}x^{\frac{3}{2}}$$
and where
$$K(x,t)=lnx[t^{6}sin(tlnx)+2t^{5}cos(tlnx)+\frac{7}{4}t^{4}sin(tlnx)+5t^{3}cos(tlnx)-\frac{21}{16}t^{2}sin(tlnx)$$
$$+\frac{9}{8}tcos(tlnx)-\frac{27}{64}sin(tlnx)]+2t^{5}cos(tlnx)-6t^{4}sin(tlnx)-3t^{3}cos(tlnx)-5t^{2}sin(tlnx)$$
$$-\frac{39}{8}tcos(tlnx)+\frac{9}{8}sin(tlnx)$$
\bigskip

By the equations \ref{ie1},\ref{ie2}, we can get a system of
integral equations

\begin{eqnarray}
\left\{
\begin{array}{ccc}
\tilde{\Delta}(x)&=&-2x^{\frac{3}{2}} \int_{0}^{\infty}K(x,t)S(t)dt+f(x)\\\\
S(t)&=&-\frac{1}{\pi}
\int_{a}^{\infty}F(t,y)\tilde{\Delta}(y)dy+g(a,t)\label{ie3}
\end{array}\right.
\end{eqnarray}

We shall prove an inverse theorem of the Guinnand formula as
follows.
\begin{theorem}
Let $f(t)$ be a function which is continous on the interval $[0,
+\infty]$ except some  discrete points,which forms a set $E$ ,and
$f(t)=lim_{X\rightarrow \infty}f_{X}(t)$,
$$f_{X}(t)=-\sum_{n<X}\frac{\Lambda(n)sin(Tlogn)}{\sqrt{n}}+\frac{\Delta(X)sin(TlogX)}{\sqrt{X}(logX)}+\int_{1}^{X}\frac{sin(Tlogy)}{\sqrt{y}logy}dy$$
and $f_{X}(t)=O(logt)$ on $[0, +\infty] \setminus E$,then the RH
holds
\end{theorem}
\begin{proof} Considering the function
$$F(s)=\int_{0}^{+\infty}\frac{2(s-1)}{(s-\frac{1}{2})^{2}+t^2}df(t)$$
and
$$f_{X}(s)=\int_{0}^{+\infty}\frac{2(s-1)}{(s-\frac{1}{2})^{2}+t^2}df_{X}(t)$$,
where $Re s > \frac{1}{2}$ Using integration by parts,
$$F(s)=-\frac{2f(0)}{s-\frac{1}{2}}+\int_{0}^{+\infty}\frac{4(s-1)t}{[(s-\frac{1}{2})^{2}+t^2]^2}f(t)dt$$
and
 $$f_{X}(s)=-\frac{2f_{X}(0)}{s-\frac{1}{2}}+\int_{0}^{+\infty}\frac{4(s-1)t}{[(s-\frac{1}{2})^{2}+t^2]^2}f_{X}(t)dt$$

Since $f(t)=lim_{X\rightarrow \infty}f_{X}(t)$ and
$f_{X}(t)=O(logt)$ on $[0, +\infty] \setminus E$, by the Lebesque
CCL,  $lim_{X\rightarrow \infty}f_{X}(s)=F(s)$
 and
\begin{eqnarray}
\int_{0}^{+\infty}\frac{4(s-1)t}{[(s-\frac{1}{2})^{2}+t^2]^2}f_{X}(t)dt=
\int_{0}^{+\infty}\frac{4(s-1)t}{[(s-\frac{1}{2})^{2}+t^2]^2}[-\sum_{n<X}\frac{\Lambda(n)sin(tlogn)}{\sqrt{n}}+\frac{\Delta(X)sin(tlogX)}{\sqrt{X}(logX)}+\int_{1}^{X}\frac{sin(Tlogy)}{\sqrt{y}logy}dy]dt
\end{eqnarray}
Using residue theorem, we have that
$$\int_{0}^{+\infty}\frac{4(s-1)t}{[(s-\frac{1}{2})^{2}+t^2]^2}\frac{\Lambda(n)sin(tlogn)}{\sqrt{n}}dt =\frac{\Lambda(n)}{n^s}$$

Similarly,
$$\int_{0}^{+\infty}\frac{4(s-1)t}{[(s-\frac{1}{2})^{2}+t^2]^2}[\frac{\Delta(X)sin(tlogX)}{\sqrt{X}(logX)}]dt=\frac{\Delta(X)}{X^{s+\frac{1}{2}}logX}$$
and
$$\int_{0}^{+\infty}\frac{4(s-1)t}{[(s-\frac{1}{2})^{2}+t^2]^2}[\int_{1}^{X}\frac{sin(Tlogy)}{\sqrt{y}logy}dy]dt=\frac{1}{1-s}[X^{1-s}-1]$$

To summarize,
\begin{eqnarray}
\int_{0}^{+\infty}\frac{4(s-1)t}{[(s-\frac{1}{2})^{2}+t^2]^2}f_{X}(t)dt=
-\sum_{n<X}\frac{\Lambda(n)}{n^s}+\frac{\Delta(X)}{X^{s+\frac{1}{2}}logX}+\frac{1}{1-s}[X^{1-s}-1]
\end{eqnarray}
Since for any $Re s>\frac{1}{2}$,$lim_{X\rightarrow
\infty}f_{X}(s)=F(s)$, and $F(s)$ is analytical in the half plane
$Res >\frac{1}{2}$, we notice that when $Re s>1$, we have
\begin{eqnarray}
lim_{X\rightarrow
\infty}f_{X}(s)=-\sum_{n=2}^{\infty}\frac{\Lambda(n)}{n^s}+\frac{1}{s-1}\\
=-\frac{\zeta'(s)}{\zeta(s)}+\frac{1}{s-1}=F(s)
\end{eqnarray}
Which means that RH is true.

\end{proof}

\bigskip
\section{Lower bound of $\tilde{\Delta}(x)$}

First of all , Let's derive a formula based on the functional
equation and formula, since

\begin{eqnarray}
\frac{\zeta'(s)}{\zeta(s)}=-s(s+1)\int_{1}^{X}\tilde{\Delta}(x)x^{-s-2}dx\nonumber\\
-\frac{s(s+1)}{2(s-1)}X^{1-s}+\sum_{\rho}\frac{s(s+1)X^{\rho-s}}{\rho(\rho+1)(s-\rho)}+\sum_{n\geq
1}\frac{s(s+1)X^{-2n-s}}{2n(2n-1)(s+2n)}\label{h1}
\end{eqnarray}

 As we know, by the
functional equation and \ref{h1}, we have
\begin{eqnarray}
Re
\frac{\zeta'(s)}{\zeta(s)}|_{s=\frac{1}{2}+it}=\frac{ln\pi}{2}-\frac{1}{2}Re\frac{\Gamma'}{\Gamma}(\frac{1}{2}+it)
\end{eqnarray}

Evaluating real part at$s=\frac{1}{2}+it$ on both sides of \ref{k1},
we get that
\begin{eqnarray}
\int_{1}^{X}\frac{\tilde{\Delta}(x)}{x^{\frac{5}{2}}}[(\frac{3}{4}-t^2)cos(tlnx)+2tsin(tlnx)]dx-
\sum_{t_{\rho}}\frac{(\frac{3}{4}-t^2)(\frac{3}{4}-t_{\rho}^2)+4tt_{\rho}}{(\frac{3}{4}-t_{\rho}^2)^2+4t_{\rho}^2}\frac{sin(t_{\rho}-t)lnX}{t_{\rho}-t}\nonumber\\
-\sum_{t_{\rho}}\frac{(\frac{3}{4}-t^2)(\frac{3}{4}-t_{\rho}^2)-4tt_{\rho}}{(\frac{3}{4}-t_{\rho}^2)^2+4t_{\rho}^2}\frac{sin(t_{\rho}+t)lnX}{t_{\rho}+t}\nonumber\\
+\sum_{t_{\rho}}\frac{\frac{3}{2}+2tt_{\rho}}{(\frac{3}{4}-t_{\rho}^2)^2+4t_{\rho}^2}cos(t_{\rho}-t)lnX\nonumber\\
+\sum_{t_{\rho}}\frac{\frac{3}{2}-2tt_{\rho}}{(\frac{3}{4}-t_{\rho}^2)^2+4t_{\rho}^2}cos(t_{\rho}+t)lnX=g_{X}(t)\label{k1}
\end{eqnarray}

Where $$g_{X}(t)=Re[-\frac{\zeta'(s)}{\zeta(s)}+\sum_{n\geq
1}\frac{s(s+1)X^{-2n-s}}{2n(2n-1)(s+2n)}]_{s=\frac{1}{2}+it}$$

By above formula, when $t\neq t_{\rho}$,

\begin{eqnarray}
\int_{1}^{X}\frac{\tilde{\Delta}(x)}{x^{\frac{5}{2}}}[(\frac{3}{4}-t^2)cos(tlnx)+2tsin(tlnx)]dx=O(1)\label{h3}
\end{eqnarray}

otherwise,
\begin{eqnarray}
\int_{1}^{X}\frac{\tilde{\Delta}(x)}{x^{\frac{5}{2}}}[(\frac{3}{4}-t_{\rho}^2)cos(t_{\rho}lnx)+2tsin(t_{\rho}lnx)]dx=lnX+O(1)\label{h2}
\end{eqnarray}

Where $\rho=\frac{1}{2}+it_{\rho}$ are Riemann zeros, to simplify
LHS of \ref{h2}, set
$\theta_{\rho}=arctan\frac{\frac{3}{4}-t_{\rho}^2}{2t_{\rho}}$,then
we have
\begin{eqnarray}
\int_{1}^{X}\frac{\tilde{\Delta}(x)}{x^{\frac{5}{2}}}sin(t_{\rho}lnx
+\theta_{\rho})dx=\frac{1}{\sqrt{(\frac{3}{4}-t_{\rho}^2)^2+4t_{\rho}^2}}lnX+O(1)\label{h4}
\end{eqnarray}
Let $u=lnx$, above formula can be reduced to
$$\int_{1}^{u}\frac{\tilde{\Delta}(e^u)}{e^{\frac{3}{2}u}}sin(t_{\rho}y
+\theta_{\rho})dy=\frac{1}{\sqrt{(\frac{3}{4}-t_{\rho}^2)^2+4t_{\rho}^2}}u+R_{\rho}(u)$$

Set
$f_{\rho}(u)=\frac{\tilde{\Delta}(e^u)}{e^{\frac{3}{2}u}}sin(t_{\rho}u
+\theta_{\rho})$, and
$A_{\rho}=\frac{1}{\sqrt{(\frac{3}{4}-t_{\rho}^2)^2+4t_{\rho}^2}}$
Thus

We have simplified form
$$\int_{0}^{X}f_{\rho}(t)dt=A_{\rho}X+R_{\rho}(X)$$

Let $g(t)=\frac{\tilde{\Delta}(e^t)}{e^{\frac{3}{2}t}}$ and
$max_{0\leq t < +\infty}\mid g(t)\mid =C_{1}$ , $\mu(x)=m\{t| \mid
g(t)\mid\leq x\}$, $\overline{\mu}(x)=X-\mu(x)$, and $E_{x}=\{t|
\mid g(t)\mid\leq x\}$

By choosing any $x<C_{1}$ ,  we have following estimate
$$A_{\rho}X+R_{\rho}(X)=\int_{0}^{X}f_{\rho}(t)dt\leq \int_{0}^{X}\mid g(t)\mid dt
=\int_{E_{x}}\mid g(t)\mid dt+\int_{[0,X]\setminus E_{x}} \mid
g(t)\mid dt$$
$$\leq x\mu(x)+C_{1}(X-\mu(x))$$

When $X$ is big enough, we have

$$\mu(x)\leq \frac{C_{1}-A_{\rho}}{C_{1}-x}X+\frac{C_{0\rho}}{C_{1}-x}$$

or
$$\overline {\mu}(x)\geq \frac{A_{\rho}-x}{C_{1}-x}X-\frac{C_{0\rho}}{C_{1}-x}$$

Where $C_{0\rho}=max_{0<t<\infty}\mid R_{\rho}(t)\mid $ and set
$F_{x}^{X}=\{u\mid |\frac{\tilde{\Delta}(u)}{u^{\frac{3}{2}}}|>x,
0<u<X\}$

Therefore $m(F_{x}^{X})\geq
\frac{A_{\rho}-x}{C_{1}-x}X-\frac{C_{0\rho}}{C_{1}-x}$

Where $\tilde{\Delta}(u)=\sum_{n\leq
u}(n-\psi(n))\Lambda(n)-\frac{u^2}{2}$

Choosing $\rho=\rho_{0}=\frac{1}{2}+14.134....$ which is the first
non-trivial Riemann zero, we have following theorem

\begin{theorem}
When $X$ is big enough, $m(F_{x}^{X})\geq
\frac{A_{\rho_{0}}-x}{C_{1}-x}X-\frac{C_{0\rho_{0}}}{C_{1}-x}$
\end{theorem}
\bigskip

Set
$$F_{X}(t)=\int_{1}^{X}\frac{\tilde{\Delta}(x)}{x^{\frac{5}{2}}}[(\frac{3}{4}-t^2)cos(tlnx)+2tsin(tlnx)]dx$$,
and $G_{X}(t)=\int_{0}^{t}F_{X}(y)dy$, by the formula
\ref{h3},\ref{h4},and Guinand formula i.e $lim_{X\rightarrow
\infty}G_{X}(t)=N(t)$, we can conjecture that when $X\rightarrow
\infty$, $F_{X}(t)$ will behave like a distribution more than an
ordinary function.  Let's verify it as follows:

First of all, we notice that
\begin{eqnarray}
\int_{0}^{+\infty}e^{-\epsilon
u}[(\frac{3}{4}-k^2)cos(ku)+2ksin(ku)]\frac{\tilde{\Delta}(e^u)}{e^{\frac{3}{2}u}}du\nonumber\\
=Re
[s(s+1)\int_{1}^{+\infty}\tilde{\Delta}(x)x^{-s-\epsilon-2}dx]_{s=\frac{1}{2}+ik}\nonumber
\end{eqnarray}

Using integration by parts couple of times, we get
\begin{eqnarray}
s(s+1)\int_{1}^{X}\tilde{\Delta}(x)x^{-s-\epsilon-2}dx=s\tilde{\Delta}(1)-\Delta(1)+\int_{1}^{X}x^{-s-\epsilon}d\Delta(x)\nonumber
+\epsilon
\int_{1}^{X}x^{-s-\epsilon}\Delta(x)dx\nonumber\\
-s\epsilon\int_{1}^{X}x^{-s-\epsilon-2}\tilde{\Delta}(x)dx-sx^{-s-\epsilon-1}\tilde{\Delta}(x)\nonumber
\end{eqnarray}
and
\begin{eqnarray}
\int_{1}^{X}x^{-s-\epsilon}d\Delta(x)=\int_{1}^{X}x^{-s-\epsilon}d\psi(x)-\int_{1}^{X}x^{-s-\epsilon}dx\nonumber
=\sum_{n<X}\frac{\Lambda(n)}{n^{s+\epsilon}}-\frac{X^{1-s-\epsilon}}{1-s-\epsilon}+1\nonumber
\end{eqnarray}
By the formula \ref{a0}, when $Res\geq \frac{1}{2}$we get that
\begin{eqnarray}
\int_{1}^{+\infty}x^{-s-\epsilon}d\Delta(x)=-\frac{\zeta'(s+\epsilon)}{\zeta(s+\epsilon)}+1
\end{eqnarray}
Therefore
\begin{eqnarray}
\int_{0}^{+\infty}e^{-\epsilon
u}[(\frac{3}{4}-k^2)cos(ku)+2ksin(ku)]\frac{\tilde{\Delta}(e^u)}{e^{\frac{3}{2}u}}
du
=Re[\frac{\zeta'(s+\epsilon)}{\zeta(s+\epsilon)}]\mid_{s=\frac{1}{2}+ik}+\varphi_{\epsilon}(k)\label{a1}
\end{eqnarray}

Where
$$\varphi_{\epsilon}(k)=Re[-s\epsilon\int_{1}^{\infty}x^{-s-\epsilon-2}\tilde{\Delta}(x)dx+\epsilon
\int_{1}^{\infty}x^{-s-\epsilon-1}\Delta(x)dx]\mid_{s=\frac{1}{2}+ik}$$

For the simplicity, we denote LHS of \ref{a1} by $J_{\epsilon}(k)$.
Choosing any test function $g(k)\in C_{0}^{\infty}(R^{+})$ ,where
$C_{0}^{\infty}(R^{+})$ is the set of all smooth functions which
have compact supports on $R^{+}$ , let's compute following inner
product,by \ref{a1}
\begin{eqnarray}
lim_{\epsilon\rightarrow 0}\langle J_{\epsilon}(k),g(k)\rangle=
Re[\frac{\zeta'(s+\epsilon)}{\zeta(s+\epsilon)}]\mid_{s=\frac{1}{2}+ik}+\varphi_{\epsilon}(k), g(k)\rangle\nonumber\\
=lim_{\epsilon\rightarrow 0}\langle
Re[\frac{\zeta'(s+\epsilon)}{\zeta(s+\epsilon)}]\mid_{s=\frac{1}{2}+ik},g(k)\rangle
+lim_{\epsilon\rightarrow 0}\langle
\varphi_{\epsilon}(k),g(k)\rangle
\end{eqnarray}

and it's not difficult to verify that $lim_{\epsilon\rightarrow
0}\langle \varphi_{\epsilon}(k),g(k)\rangle=0$

Using integration by parts twice, we have that
\begin{eqnarray}
\langle
Re[\frac{\zeta'(s+\epsilon)}{\zeta(s+\epsilon)}]\mid_{s=\frac{1}{2}+ik},g(k)\rangle=\langle
h(\frac{1}{2}+\epsilon+ik), g''(k) \rangle
\end{eqnarray}

Where $h(z)=\int_{1}^{z}ln\zeta(s)ds$ and $Re z>\frac{1}{2}$,the
integration path is the conventional contour from $1$ to $z$

Since $h(\frac{1}{2}+\epsilon+ik)=O(klogk)$ uniformly for any small
$\epsilon$ and $lim_{\epsilon\rightarrow
0}h(\frac{1}{2}+\epsilon+ik)=S_{1}(k)$

Where $S_{1}(k)=\int_{0}^{k}S(t)dt$

Finally, we have
\begin{eqnarray}
lim_{\epsilon\rightarrow 0}\langle
J_{\epsilon}(k),g(k)\rangle=\langle S_{1}(k), g''(k) \rangle
\end{eqnarray}

Before ending this section, let's take look at the equation \ref{k1}
again, we can rewrite this equation in term of integral equation as
follows:
\begin{eqnarray}
\int_{0}^{\infty}K_{X}(t,t')dN(t')=H_{X}(t)
\end{eqnarray}
Where
$$K_{X}(t)=-
\frac{(\frac{3}{4}-t^2)(\frac{3}{4}-t'^{2})+4tt'}{(\frac{3}{4}-t'^{2})^2+4t'^{2}}\frac{sin(t'-t)lnX}{t'-t}
-\frac{(\frac{3}{4}-t^2)(\frac{3}{4}-t'^{2})-4tt'}{(\frac{3}{4}-t'^{2})^2+4t'^{2}}\frac{sin(t'+t)lnX}{t'+t}$$
$$+\frac{\frac{3}{2}+2tt'}{(\frac{3}{4}-t'^{2})^2+4t'^{2}}cos(t'-t)lnX+\frac{\frac{3}{2}-2tt'}{(\frac{3}{4}-t'^{2})^2+4t'^{2}}cos(t'+t)lnX$$

and
$$H_{X}(t)=\int_{1}^{X}\frac{\tilde{\Delta}(x)}{x^{\frac{5}{2}}}[(\frac{3}{4}-t^2)cos(tlnx)+2tsin(tlnx)]dx+\frac{ln\pi}{2}-\frac{1}{2}Re\frac{\Gamma'}{\Gamma}(\frac{1}{2}+it)$$

Actually, above equation depends on the parameter $X$,for every
fixed $X$,we get a non trivial integral equation of $N(t)$, so we
obtain a family of integral equations, noticing that the integral
kernel $K_{X}(t,t')$ is an explicit function,it's expected that
exploring these integral equations will help us to understand RH
further, besides we can consider similar results for $L$ function
which satisfies functional equation.

\end{document}